\documentclass[12pt]{amsart}

\usepackage[latin1]{inputenc}
\usepackage{amssymb,amsmath,amsfonts, amsthm}
\usepackage{enumerate}
\usepackage{times}
\usepackage{bbding}
\usepackage[margin=2.0cm]{geometry}
\usepackage{cite}
\newtheorem{theorem}{Theorem}[section]
\newtheorem{lemma}[theorem]{Lemma}
\newtheorem{proposition}[theorem]{Proposition}

\numberwithin{equation}{section}
\newcommand{\fconvs}{{\mbox{\rm Conv}_{{\rm sc}}(\R^n)}} 
\newcommand{\fconvsk}{{\mbox{\rm Conv}_{{\rm sc}}(\R^k)}} 
\newcommand{\fconvso}{{\mbox{\rm Conv}_{{\rm sc},0}(\R^n)}} 
\newcommand{\fconvx}{{\mbox{\rm Conv}(\R^n)}} 
\newcommand{\fconvf}{{\mbox{\rm Conv}(\R^n; \R)}} 
\newcommand{\fconvsE}{{\mbox{\rm Conv}_{{\rm sc}}(E)}} 
\newcommand{\fconvsbE}{{\mbox{\rm Conv}_{{\rm sc}}(\bar{E})}} 
\newcommand{\infconv}{\mathbin{\Box}} 
\newcommand{\sq}{\mathbin{\vcenter{\hbox{\rule{.3ex}{.3ex}}}}} 
\newcommand{\proj}{\operatorname{proj}} 
\newcommand{\bd}{\operatorname{bd}} 

\newcommand{\Grass}[2]{\operatorname{G}(#2,#1)} 

\newcommand{\cK}{{\mathcal K}}

\newcommand{\cR}{\operatorname{\mathcal R}}

\newcommand{\sn}{{\mathbb{S}^{n-1}}} 
\newcommand{\Bn}{B^n} 

\renewcommand{\O}{\operatorname{O}}
\newcommand{\SO}{\operatorname{SO}}

\newcommand{\Hess}{{\operatorname{D}}^2}

\DeclareMathOperator{\oZ}{\operatorname{Z}}

\newcommand{\oZZ}[2]{\operatorname{V}_{#1,#2}} 

\newcommand{\R}{{\mathbb R}}
\newcommand{\N}{{\mathbb N}}

\newcommand{\dom}{\operatorname{dom}}
\newcommand{\hm}{\mathcal H}
\newcommand{\epi}{\operatorname{epi}}
\newcommand{\ind}{{\rm\bf I}}

\newcommand{\Had}[2]{D_{#1}^{#2}} 

\renewcommand{\d}{\,\mathrm{d}}

\begin{document}

\title[The Hadwiger Theorem on Convex Functions, II]{The Hadwiger Theorem on Convex Functions, II:\\
Cauchy--Kubota Formulas}

\author{Andrea Colesanti}
\address{Dipartimento di Matematica e Informatica ``U. Dini''
Universit\`a degli Studi di Firenze,
Viale Morgagni 67/A - 50134, Firenze, Italy}
\email{andrea.colesanti@unifi.it}

\author{Monika Ludwig}
\address{Institut f\"ur Diskrete Mathematik und Geometrie,
Technische Universit\"at Wien,
Wiedner Hauptstra\ss e 8-10/1046,
1040 Wien, Austria}
\email{monika.ludwig@tuwien.ac.at}

\author{Fabian Mussnig}
\address{Dipartimento di Matematica e Informatica ``U. Dini''
Universit\`a degli Studi di Firenze,
Viale Morgagni 67/A - 50134, Firenze, Italy}
\email{mussnig@gmail.com}

\date{}

\begin{abstract} A new version of the Hadwiger theorem on convex functions is established and an explicit representation of functional intrinsic volumes is found using new functional Cauchy--Kubota formulas.  In addition, connections between functional intrinsic volumes and their classical counterparts are obtained and non-negative valuations are classified. 

\end{abstract}

\maketitle

\section{Introduction and Statement of Results}
Valuations play a central role in convex and integral geometry ever since they were the key ingredient in  Dehn's solution of Hilbert's Third Problem in 1901 (see \cite{Hadwiger:V,Klain:Rota}).  
In the classical setting, valuations  are defined on the set of convex bodies, $\cK^n$, that is, on non-empty, compact, convex subsets of $\R^n$, and a functional $\oZ:\cK^n\to\R$ is called a \emph{valuation} if
$$\oZ(K)+\oZ(L)=\oZ(K\cup L) + \oZ(K\cap L)$$
for every $K,L\in \cK^n$ such that $K\cup L\in\cK^n$. Among the most important valuations are the intrinsic volumes, $V_j$, for $0\leq j \leq n$. 
Here, $V_n$ is the $n$-dimensional volume (or Lebesgue measure) and $V_0$ is the Euler characteristic (that is, $V_0(K):=1$ for every $K\in\cK^n$). 
If $K\in\cK^n$ is $j$-dimensional for $1\leq j \leq n-1$, then $V_j(K)$ is just the $j$-dimensional volume of $K$. For general $K\in\cK^n$, the $j$th intrinsic volume of $K$ can be defined using the Cauchy--Kubota formulas
\begin{equation}\label{cauchy_kubota}
V_j(K):=\frac{\kappa_n}{\kappa_j \kappa_{n-j}} \binom{n}{j} \int_{\Grass{j}{n}} V_j(\proj_E K)\d E.
\end{equation}
Here, $\kappa_j$ denotes the $j$-dimensional volume of the $j$-dimensional unit ball, integration is with respect to the Haar probability measure on $\Grass{j}{n}$, the Grassmannian of $j$-dimensional subspaces in $\R^n$, and $\proj_E:\R^n\to E$ denotes the orthogonal projection onto $E\in \Grass{j}{n}$ (cf. \cite{Schneider:CB2}).

While \eqref{cauchy_kubota} can be proved directly, it is also an immediate consequence of the celebrated Hadwiger 
theorem, which classifies continuous, translation and rotation invariant valuations and thereby characterizes linear combinations of intrinsic volumes. Here, continuity is understood with respect to the Hausdorff metric, and a valuation $\oZ:\cK^n\to\R$ is \emph{translation invariant} if $\oZ(\tau K)=\oZ(K)$ for every $K\in\cK^n$ and translation $\tau$ on $\R^n$, while it is \emph{rotation invariant} if $\oZ(\vartheta K)=\oZ(K)$ for every $K\in\cK^n$ and $\vartheta\in \SO(n)$.

\begin{theorem}[Hadwiger \cite{Hadwiger:V}]\label{hugo}
A functional $\oZ:\cK^n\to \R$ is  a continuous, translation and rotation invariant valuation if and only if  there exist constants $\zeta_0, \ldots, \zeta_n\in\R$ such that 
$$\oZ(K) = \sum_{j=0}^n \zeta_j \,V_j(K)$$
for every $K\in\cK^n$. 
\end{theorem}

\noindent
The Hadwiger theorem leads to effortless proofs of numerous further results in integral geometry and geometric probability (see \cite{Hadwiger:V,Klain:Rota}).

\goodbreak
We restate the Hadwiger theorem here in a form that makes use of the Cauchy--Kubota formulas (\ref{cauchy_kubota}). 

\begin{theorem}[Hadwiger]
\label{thm2:hadwiger}
A functional $\,\oZ:\cK^n \to \R$ is a continuous, translation and rotation invariant valuation if and only if there exist constants $\alpha_0, \dots, \alpha_n\in\R$  such that
\begin{equation*}
\oZ(K)=\sum_{j=0}^n \alpha_j \int_{\Grass{j}{n}} V_j(\proj_E K)\d E
\end{equation*}
for every $K\in\cK^n$.
\end{theorem}

\noindent
The Hadwiger theorem is the first culmination of the program, initiated by Blaschke, of classifying valuations invariant under various groups  and the starting point of  geometric  valuation theory (see \cite[Chapter~6]{Schneider:CB2}). We refer to \cite{Alesker99,Alesker01,AleskerFaifman, BernigFaifman2017,Bernig:Fu, Haberl_sln, LiMa,  Ludwig:Reitzner,Ludwig:Reitzner2,Ludwig:convex,Haberl:Parapatits_centro} for some recent classification results and to \cite{HLYZ_acta, BLYZ_cpam, Ludwig:matrix} for some of  the new valuations that keep arising.

Currently, a geometric theory of valuations on function spaces is being developed.
On a space $X$ of (extended) real-valued functions, a functional $\oZ:X\to\R$ is called a \emph{valuation} if
$$\oZ(u)+\oZ(v)=\oZ(u\vee v) + \oZ(u \wedge v)$$
for every $u,v\in X$ such that also their pointwise maximum $u\vee v$ and their pointwise minimum $u\wedge v$ belong to $X$. 
The first classification results of valuations on classical function spaces were obtained for $L_p$ and Sobolev spaces and for Lipschitz and continuous functions (see  \cite{Tsang:Lp, Ludwig:SobVal,Ludwig:Fisher,ColesantiPagniniTradaceteVillanueva,ColesantiPagniniTradaceteVillanueva2021, Villanueva2016}). 

Of special interest are valuations on convex functions, where the first classification results were obtained in \cite{ColesantiLudwigMussnig17,ColesantiLudwigMussnig,Mussnig19, Mussnig21} and the first structural results in  \cite{Alesker_cf,Colesanti-Ludwig-Mussnig-4, Knoerr1, Knoerr2}.
Recently, the authors \cite{Colesanti-Ludwig-Mussnig-5} established the Hadwiger theorem on convex functions. 
Let
$$\fconvs:=\Big\{u:\R^n\to(-\infty,+\infty]\colon u \not\equiv +\infty, \lim_{|x|\to+\infty} \frac{u(x)}{|x|}=+\infty, u \text{ is l.s.c. and convex}\Big\}$$
denote the space of proper,  super-coercive, lower semicontinuous, convex functions on $\R^n$, where $\vert\cdot\vert$ is the Euclidean norm.
It is equipped with the topology induced by epi-convergence (see Section \ref{convex functions} for the definition). A functional $\oZ:\fconvs\to\R$ is \emph{epi-translation invariant} if $\oZ(u\circ \tau^{-1}+\alpha)=\oZ(u)$ for every $u\in\fconvs$, every translation $\tau$ on $\R^n$ and every $\alpha\in\R$. It is \emph{rotation invariant} if $\oZ(u\circ \vartheta^{-1})=\oZ(u)$ for every $u\in\fconvs$ and $\vartheta\in\SO(n)$.

The authors \cite{Colesanti-Ludwig-Mussnig-5} introduced functional versions of  intrinsic volumes on $\fconvs$ in the following way.
For $0\leq j \leq n-1$, let
$$
\Had{j}{n}:=\Big\{\zeta\in C_b((0,\infty))\colon  \lim_{s\to 0^+} s^{n-j} \zeta(s)=0,  \lim_{s\to 0^+} \int_s^{\infty}  t^{n-j-1}\zeta(t) \d t \text{ exists and is finite}\Big\},
$$
where  $C_b((0,\infty))$ is the set of continuous functions with bounded support on $(0,\infty)$. In addition, let $\zeta\in\Had{n}{n}\,$ if $\zeta\in C_b((0,\infty))$ and $\lim_{s\to 0^+} \zeta(s)$ exists and is finite. In this case, we set $\zeta(0):=\lim_{s\to 0^+} \zeta(s)$ and consider $\zeta$ also as an element of $C_c([0,\infty))$, the set of continuous functions with compact support on $[0,\infty)$.

\begin{theorem}[\!\cite{Colesanti-Ludwig-Mussnig-5}, Theorem 1.2]
\label{thm:existence_singular_hessian_vals}
For $0\leq j \leq n$ and $\zeta\in\Had{j}{n}$, there exists a unique, continuous, epi-translation and rotation invariant valuation $\oZ:\fconvs\to\R$ such that
\begin{equation}
\label{eq:rep_ozz_c2}
\oZ(u)=\int_{\R^n} \zeta(|\nabla u(x)|)\big[\Hess u(x)\big]_{n-j} \d x
\end{equation}
for every $u\in\fconvs\cap C_+^2(\R^n)$.
\end{theorem}

\noindent
Here $C_+^2(\R^n)$ is the set of finite-valued functions  $u\in C^2(\R^n)$ with positive definite Hessian matrix $\Hess u$ and we write $[A]_i$ for the $i$th elementary symmetric function of the eigenvalues of any symmetric matrix $A$ (with the convention that $[A]_0:=1$).
\goodbreak

Theorem \ref{thm:existence_singular_hessian_vals} allows us to make the following definition. For $0\leq j \leq n$ and $\zeta\in\Had{j}{n}$, the \emph{functional intrinsic volume} $\,\oZZ{j}{\zeta}^n: \fconvs\to\R$ is the unique continuous extension of the functional defined in \eqref{eq:rep_ozz_c2} on $\fconvs\cap C_+^2(\R^n)$. Note that the functional intrinsic volumes are not only rotation invariant but even $\O(n)$ invariant. Moreover,  for $\zeta\in\Had{0}{n}$, the functional $\oZZ{0}{\zeta}^n$ is a constant, independent of $u\in\fconvs$, and by
\cite[Proposition 20]{Colesanti-Ludwig-Mussnig-4},
\begin{equation}
\label{n-hom}
\oZZ{n}{\zeta}^n(u)=\int_{\dom u} \zeta(|\nabla u(x)|)\d x
\end{equation}
for every $u\in\fconvs$ and $\zeta\in\Had{n}{n}$, where $\dom u:=\{x\in\R^n: u(x)<\infty\}$ is the \emph{domain} of $u$. We remark that for $\zeta\in C_c([0,\infty))$, 
extensions of \eqref{eq:rep_ozz_c2} to $\fconvs$ were previously defined by the authors in \cite{Colesanti-Ludwig-Mussnig-4} using Hessian measures and so-called Hessian valuations. For the proof of Theorem~\ref{thm:existence_singular_hessian_vals} in \cite{Colesanti-Ludwig-Mussnig-5}, singular Hessian valuations were introduced.

The Hadwiger theorem for convex functions is the following result.
Let $n\geq 2$.
\begin{theorem}[\!\!\cite{Colesanti-Ludwig-Mussnig-5}, Theorem 1.3]
\label{thm:hadwiger_convex_functions}
A functional $\oZ:\fconvs \to \R$ is a continuous, epi-translation and rotation invariant valuation if and only if there exist functions $\zeta_0\in\Had{0}{n}$, \dots, $\zeta_n\in\Had{n}{n}$  such that
\begin{equation*}
\oZ(u)= \sum_{j=0}^n \oZZ{j}{\zeta_j}^n(u) 
\end{equation*}
for every $u\in\fconvs$.
\end{theorem}

\noindent
Theorem~\ref{hugo} and Theorem~\ref{thm:hadwiger_convex_functions} show that the functionals $\oZZ{j}{\zeta}^n$ clearly play the role of intrinsic volumes on $\fconvs$.

In this article, we present an integral-geometric approach to valuations on convex functions. We obtain a new version of the Hadwiger theorem on $\fconvs$, Theorem \ref{thm2:hadwiger_convex_functions_ck}, based on new functional Cauchy--Kubota formulas, and we present a new proof of Theorem~\ref{thm:existence_singular_hessian_vals}.
First, we establish a new integral-geometric representation of the functionals $\oZZ{j}{\zeta}^n$, corresponding to the Cauchy--Kubota formulas \eqref{cauchy_kubota}.
For a linear subspace $E\subseteq \R^n$, we write $\fconvsE$ for the set of proper, lower semicontinuous, super-coercive, convex functions $w:E\to (-\infty, +\infty]$. For $u\in\fconvs$, define the \emph{projection function} $\proj_E u: E\to (-\infty,\infty]$ by
$$\proj_E u(x_E) := \min\nolimits_{z\in E^\perp} u(x_E+z)$$
for $x_E\in E$, where $E^\perp$ denotes the orthogonal complement of $E$. 
If $\oZ:\fconvsk\to\R$ is $\O(k)$ invariant and $\dim E=k$, we define $\oZ$ on $\fconvsE$ by identifying $\fconvsE$ with $\fconvsk$ 
(see Section~\ref{se:proj_fcts}).

\begin{theorem}
	\label{thm:cauchy_kubota_for_hessian_vals}
	Let $\,0\leq j \leq k<n$. If $\zeta\in\Had{j}{n}$, then
	\begin{equation}
	\label{eq:cauchy_kubota_via_hadwiger}
	\oZZ{j}{\zeta}^n(u) =   \frac{\kappa_n}{\kappa_k \kappa_{n-k} }\binom{n}{k} \int_{\Grass{k}{n}}  \oZZ{j}{\xi}^k(\proj_E u) \d E 
	\end{equation}
	for every $u\in\fconvs$, where $\xi\in\Had{j}{k}$ is given by
    \begin{equation}
    \label{eq:def_xi}
    \xi(s):= \frac{\kappa_{n-k}}{\binom{n-j}{k-j}} \big(s^{n-k}\zeta(s) +(n-k) \int_s^{\infty}  t^{n-k-1}\zeta(t)\d t\big)
    \end{equation}
for $s>0$.
\end{theorem}
\noindent
Here we set $\kappa_0:=1$ and $\Had{0}{0}:=\Had{1}{1}$. Further, let $\oZZ{0}{\xi}^0(\proj_E u) := \xi(0)$ for $\xi \in\Had{0}{0}$.

\goodbreak
In the proof of this theorem we make essential use of results from \cite{Colesanti-Ludwig-Mussnig-5} that were established for the proof of Theorem~\ref{thm:existence_singular_hessian_vals}. We also use tools from the integral geometry of convex bodies. The proof of Theorem \ref{thm:cauchy_kubota_for_hessian_vals} and our new proof of Theorem~\ref{thm:existence_singular_hessian_vals} are presented in Section~\ref{se:ck_direct}. Note that embedding $\cK^n$ into $\fconvs$, we see that \eqref{eq:cauchy_kubota_via_hadwiger} generalizes the classical Cauchy--Kubota formulas (see Section~\ref{se:ret_intr_vols}).

\goodbreak
As a consequence of Theorem~\ref{thm:cauchy_kubota_for_hessian_vals} (with $j=k$) and the representation of the functional intrinsic volume  
for $j=n$ in (\ref{n-hom}), we immediately obtain the following representation of $\oZZ{j}{\zeta}^n$ for $0\le j<n$. This is the first explicit representation of functional intrinsic volumes as integrals, as in \cite{Colesanti-Ludwig-Mussnig-5} limits using Moreau--Yosida approximation were used.

\begin{theorem}\label{cauchy_function}
Let $\,0\le j<n$. If $\zeta\in\Had{j}{n}$,  then
\begin{equation*}
\oZZ{j}{\zeta}^n(u)= \frac{\kappa_n}{\kappa_j\kappa_{n-j}} \binom{n}{j} \int_{\Grass{j}{n}}  \int_{\dom (\proj_E u)} \alpha(|\nabla \proj_E u(x_E)|) \d x_E \d E
\end{equation*}
for every $u\in\fconvs$, where $\alpha\in C_c([0,\infty))$ is given by 
$$\alpha(s):=  \kappa_{n-j} \big(  s^{n-j}\zeta(s)+(n-j)\int_s^{\infty}  t^{n-j-1} \zeta(t)\d t\big)$$
for $s>0$. 
\end{theorem}
\noindent
Here, in the case $j=0$, we set $\nabla\proj_E u(x_E):=0$ and $\oZZ{0}{\zeta}^n(u):=\alpha(0)$ for $u\in\fconvs$. Note that a convex function is differentiable almost everywhere on the interior of its domain and hence the integral representing $\oZZ{j}{\zeta}^n(u)$ is well-defined for $u\in\fconvs$.

\goodbreak
Theorem \ref{thm:hadwiger_convex_functions} and Theorem \ref{cauchy_function} imply the following new version of the Hadwiger theorem for convex functions, which corresponds to Theorem~\ref{thm2:hadwiger}. Let $n\geq 2$.

\begin{theorem}
\label{thm2:hadwiger_convex_functions_ck}
A functional $\oZ:\fconvs \to \R$ is a continuous, epi-translation and rotation invariant valuation if and only if there exist  functions $\alpha_0, \dots, \alpha_n \in C_c([0,\infty))$ such that
\begin{equation*}
\oZ(u)=  \sum_{j=0}^{n}\int_{\Grass{j}{n}}  \int_{\dom (\proj_E u)} \alpha_j(|\nabla \proj_E u(x_E)|) \d x_E \d E 
\end{equation*}
for every $u\in\fconvs$. 
\end{theorem}

\noindent
Note that by Theorem \ref{cauchy_function} and properties of the integral transform which maps $\zeta$ to $\alpha$ (see Lemma \ref{le:r_kln}), Theorem \ref{thm2:hadwiger_convex_functions_ck}  is in fact equivalent to Theorem~\ref{thm:hadwiger_convex_functions}.

In Section~\ref{se:dual}, we present results for valuations on
$\fconvf\!\!:=\{v:\R^n\to \R\colon v \text{ is convex}\}$, the space of finite-valued convex functions. The results are obtained from results for valuations on $\fconvs$ by using the Legendre--Fenchel transform or convex conjugate. The new Cauchy--Kubota formulas correspond to results on restrictions of convex functions to linear subspaces in this setting.

In the final section, we collect several applications and results. In particular, 
we present a second proof of Theorem~\ref{thm:cauchy_kubota_for_hessian_vals} which uses Theorem~\ref{thm:hadwiger_convex_functions}. Thus, similar to the classical Cauchy--Kubota formulas \eqref{cauchy_kubota}, Theorem~\ref{thm:cauchy_kubota_for_hessian_vals} can be  proved both directly and as a consequence of the Hadwiger theorem. 
We also obtain connections between functional intrinsic volumes and their classical counterparts
and answer questions about non-negative and monotone valuations.

\section{Preliminaries}
We work in $n$-dimensional Euclidean space $\R^n$, with $n\ge 1$, endowed with the Euclidean norm $\vert \cdot\vert $ and the standard scalar product
$\langle \cdot,\cdot\rangle$. We also use coordinates, $x=(x_1,\dots,x_n)$, for $x\in\R^n$. 
For $k\le n$, we often identify $\R^k$ with $
\{x\in\R^n\colon x_{k+1}=\dots=x_n=0\}$.
Let $\Bn\!:=\{x\in\R^n: \vert x\vert \le 1\}$ be the Euclidean unit ball and $\sn$ the unit sphere in $\R^n$.

\subsection{Convex Bodies}\label{convex bodies} A basic reference on convex bodies is the book by Schneider \cite{Schneider:CB2}. For $K\in\cK^n$, its {\em support function} $h_K:\R^n\to \R$ is defined as
$$h_K(x):=\max\nolimits_{y\in K}\langle x,y \rangle.$$
It is a one-homogeneous and convex function that determines $K$. 

For $K\in\cK^n$ and $0\leq j \leq n-1$, let $C_j(K,\cdot)$ be its $j$th curvature measure (see \cite{Schneider:CB2}). We require the following integral-geometric formula.
Let $0\le j\le k<n$. By (4.79) in \cite{Schneider:CB2}, for every $K\in\cK^n$ and every Borel set $B \subseteq \bd K$,  we have
\begin{equation}\label{Schneider1}
C_{j}(K, B)=\frac{n \kappa_n}{k \kappa_k} \int_{\Grass{k}{n}}C^E_{j}(\proj_E K,\proj_E B)\d E,
\end{equation}
where  $C^E_j(\proj_E K,\cdot)$ is the $j$th curvature measure of the convex body $\proj_E K$ taken with respect to the
subspace $E$ and $\bd K$ is the boundary of $K$. 

Under suitable regularity assumptions, curvature measures can be expressed in terms of the principal curvatures of the boundary. Let $K\in\cK^n$ have boundary of class $C^2$ with positive Gauss curvature. For $0\leq j \leq n-1$ and $x\in\bd K$, let $\tau_j(K,x)$ be the $j$th elementary symmetric function of the principal curvatures of $\bd K$ at $x$. By (2.36) and (4.25) in \cite{Schneider:CB2}, we have
\begin{equation}\label{curvature measures smooth case}
C_j(K,B)=\binom{n-1}{n-1-j}^{-1}\int_B\tau_{n-1-j}(K,x)\d\hm^{n-1}(x)
\end{equation}
for every $0\leq j \leq n-1$ and for every Borel set $B\subseteq\bd K$,  where $\hm^k$ is the $k$-dimensional Hausdorff measure.

\subsection{Convex Functions}\label{convex functions}
We collect some basic results and properties of convex functions. Standard references are the books by Rockafellar \cite{Rockafellar} and Rockafellar \& Wets \cite{RockafellarWets} (also, see \cite{ColesantiLudwigMussnig}).

Let $\fconvx$ be the set of proper, lower semicontinuous, convex functions $u:\R^n\to(-\infty,\infty]$. 
Every function $u\in\fconvx$ is uniquely determined by its \emph{epi-graph}
$$\epi u := \{(x,t)\in \R^n\times \R\colon u(x)\leq t\},$$
which is a closed, convex subset of $\R^{n+1}$. 
For $t\in\R$, we write
$$\{u< t\}:=\{x\in\R^n\colon u(x)<t\},\quad\quad \{u\leq t\}:=\{x\in\R^n\colon u(x)\leq t\}$$
for the \emph{sublevel sets} of $u$, which are convex subsets of $\R^n$. Since $u$ is lower semicontinuous, the sublevel sets $\{u\leq t\}$ are closed. If in addition $u\in\fconvs$, then the sublevel sets are bounded.
Similarly, we write
$$\{u=t\}:=\{x\in\R^n\colon u(x)=t\},\quad\quad \{t_1 < u \leq t_2\}:=\{x\in\R^n\colon t_1 < u(x) \leq t_2\}$$
for $t\in\R$ and  $t_1<t_2$.

The standard topology on $\fconvx$ and its subsets is induced by epi-convergence. A sequence of functions $u_k\in\fconvx$ is \emph{epi-convergent} to $u\in\fconvx$ if for every $x\in\R^n$:
\begin{enumerate}
    \item[(i)] $u(x)\leq \liminf_{k\to\infty} u_k(x_k)$ for every sequence $x_k\in\R^n$ that converges to $x$;
    \item[(ii)] $u(x)=\lim_{k\to\infty} u_k(x_k)$ for at least one sequence $x_k\in\R^n$ that converges to $x$.
\end{enumerate}
Note that the limit of an epi-convergent sequence of functions from $\fconvx$ is always lower semi\-continuous. 

A sequence of functions $v_k\in\fconvf$ is epi-convergent to $v\in\fconvf$ if and only if $v_k$ converges pointwise to $v$, which by convexity is equivalent to uniform convergence on compact sets. 
On $\fconvs$, epi-convergence is, basically,  equivalent to Hausdorff convergence of sublevel sets. Here we say that for a sequence $u_k\in\fconvs$, the sets $\{u_k\leq t \}$ converge to the empty set, if there exists $k_0\in\N$ such that $\{u_k\leq t \}=\emptyset$ for every $k\geq k_0$.

\begin{lemma}
\label{le:hd_conv_lvl_sets}
Let $u_k,u\in\fconvs$. If $u_k$ epi-converges to $u$, then $\{u_k\leq t \}$ converges to $\{u\leq t\}$ for every $t\neq \min_{x\in\R^n} u(x)$. Conversely, if for every $t\in\R$ there exists a sequence $t_k\to t$ such that $\{u_k\leq t_k\}$ converges to $\{u\leq t\}$, then $u_k$ epi-converges to $u$.
\end{lemma}

For $u\in\fconvx$, let $u^*\in\fconvx$ be its \emph{Legendre--Fenchel transform} or \emph{convex conjugate}, which is defined by
$$u^*(y):=\sup\nolimits_{x\in\R^n} \big(\langle x,y \rangle - u(x) \big)$$
for $y\in\R^n$. Since $u$ is lower semicontinuous, $u^{**}=u$. Moreover, $u\in\fconvs$ if and only if $u^*\in\fconvf$, and $u\in \fconvs\cap C_+^2(\R^n)$ if and only if $u^*\in \fconvs\cap C_+^2(\R^n)$.

\begin{lemma}\label{wijsman}
A sequence of functions $u_k$ in $\fconvx$ is epi-convergent to $u\in\fconvx$ if and only if $u_k^*$ is epi-convergent to $u^*$.
\end{lemma}

\noindent
Since $\fconvs\cap C_+^2(\R^n)$ is dense in $\fconvf$, this implies the following simple result.

\begin{lemma}
\label{le:C2p_dense}
For every $u\in\fconvs$, there exists a sequence of functions from $\fconvs\cap C_+^2(\R^n)$ that epi-converges to $u$.
\end{lemma}
\goodbreak

For a convex body $K\in\cK^n$, let
$$
\ind_K(x):=\begin{cases}
0\quad &\text{if } x\in K,\\
+\infty\quad &\text{if } x\not\in K
\end{cases}
$$
be its \emph{(convex) indicator function}. Clearly, $\ind_K\in \fconvs$ while $\ind_K^*=h_K$ and $h_K\in\fconvf$.

\goodbreak
For $u\in\fconvx$, the \emph{subdifferential} of $u$ at $x\in\R^n$ is defined by
$$\partial u(x) := \{y\in\R^n\colon u(z)\geq u(x)+\langle y, z-x\rangle \text{ for } z\in\R^n \}.$$
Every element of $\partial u(x)$ is called a \emph{subgradient} of $u$ at $x$. If $u$ is differentiable at $x$, then $\partial u(x)=\{\nabla u(x)\}$. For $x,y\in\R^n$, we have $y\in\partial u(x)$ if and only if $x\in\partial u^*(y)$.

For functions $u_1,u_2\in\fconvs$, we denote by $u_1\infconv u_2\in\fconvs$ their \emph{infimal convolution} which is defined as
$$(u_1\infconv u_2)(x):=\inf\nolimits_{x_1+x_2=x} u_1(x_1)+u_2(x_2)$$
for $x\in\R^n$. Note that
$$\epi (u_1\infconv u_2)=\epi u_1 + \epi u_2,$$
where the addition on the right side is the Minkowski addition of subsets in $\R^{n+1}$. Further, we define \emph{epi-multiplication} on $\fconvs$ in the following way. For $\lambda>0$ and $u\in\fconvs$, let
$$\lambda\sq u(x):=\lambda\, u\left( \frac x\lambda \right)$$
for $x\in\R^n$. This corresponds to rescaling the epi-graph of $u$ by the factor $\lambda$, that is, $\epi \lambda\sq u=\lambda \epi u$.

The two operations above can also be described using convex conjugates. For $u_1,u_2\in\fconvs$, we have
$$(u_1\infconv u_2)^* = u_1^* + u_2^*,$$
where the addition on the right side is the pointwise addition of functions. Similarly,
$$(\lambda\sq u)^* = \lambda\,u^*$$
for every $u\in\fconvs$ and $\lambda>0$.

\subsection{Hessian Measures}
We will use two families of Hessian measures of convex functions. For a more detailed presentation, see \cite{ColesantiHug2000,Colesanti-Ludwig-Mussnig-3}.
We remark that Hessian measures were introduced by Trudinger and Wang \cite{TrudingerWang1997, Trudinger:Wang1999} in the context of so-called Hessian equations.

For $u\in\fconvs$, we use the non-negative Borel measures $\Psi^n_j(u,\cdot)$ for $0\leq j\leq n$ that have the  property that for every Borel function $\beta:\R^n\to [0,\infty)$,
\begin{equation*}
\int_{\R^n} \beta(y) \d\Psi^n_j(u,y)=\int_{\R^n} \beta(\nabla u(x)) \big[\Hess u(x)\big]_{n-j}\d x
\end{equation*}
for $u\in\fconvs\cap C_+^2(\R^n)$. In addition, 
\begin{equation}\label{n-gradient}
\int_{\R^n} \beta(y) \d\Psi^n_n(u,y)= \int_{\dom u} \beta(\nabla u(x))\d x
\end{equation}
for $u\in\fconvs$ and $\beta \in C_c(\R^n)$. For $v\in\fconvf$, we use the non-negative Borel measures $\Phi^n_j(v,\cdot)$ for $0\leq j \leq n$ that have the  property that for every Borel function $\beta:\R^n\to [0,\infty)$,
\begin{equation*}
\int_{\R^n} \beta(x) \d\Phi^n_j(v,x)=\int_{\R^n} \beta(x) \big[\Hess v(x)\big]_{j}\d x
\end{equation*}
for $v\in\fconvf\cap C_+^2(\R^n)$. The measure $\Phi_n^n(v,\cdot)$ is called the Monge--Amp\`ere measure of $v$.

The interplay of Hessian measures and convex conjugation is well understood. Let $u\in\fconvs$ and $0\leq j \leq n$. It is an immediate consequence of \cite[Theorem 8.2]{Colesanti-Ludwig-Mussnig-3} that 
\begin{equation}
\label{eq:int_u_psi_int_v_phi}
\int_{B} \beta(y) \d\Psi_j^n(u,y) = \int_{B} \beta(x) \d\Phi_j^n(u^*,x)
\end{equation}
for every $u\in\fconvs$ and  Borel subset $B\subseteq \R^n$, when $\beta:\R^n\backslash\{0\}\to\R$ is such that one of the two integrals above, and therefore both, exist.

\subsection{Valuations on Convex Functions}
\label{se:vals_on_convex_fcts}

We say that $\oZ:\fconvs\to\R$ is \emph{epi-homo\-geneous} of degree~$j$ if $\oZ(\lambda\sq u)=\lambda^j\,\oZ(u)$ for every $u\in\fconvs$ and $\lambda>0$.

The following result is an immediate consequence of \cite[Proposition 20]{Colesanti-Ludwig-Mussnig-4}. 

\begin{proposition}
\label{prop:n-hom}
For $\zeta\in C_c([0,\infty))$, the functional $\oZ:\fconvs \to \R$, defined by
$$\oZ(u):=\int_{\dom (u)} \zeta(|\nabla u(x)|) \d x,$$
is a continuous, epi-translation and $\O(n)$
invariant valuation that is epi-homogeneous of degree $n$.
\end{proposition}

Next, we consider  valuations on $\fconvf$. 
For $X\subseteq\fconvx$, we associate with a valuation $\oZ:X\to\R$ its \emph{dual valuation} $\oZ^*$ defined on $X^*:=\{u^*: u\in X\}$ by setting
$$\oZ^*(u):=\oZ(u^*).$$
It was shown in \cite{Colesanti-Ludwig-Mussnig-3} that $\oZ:X\to\R$ is a continuous valuation if and only if
$\oZ^*:X^*\to\R$ is a continuous valuation. Since $u\in\fconvs$ if and only if $u^*\in\fconvf$, this allows us to transfer results between $\fconvs$ and $\fconvf$. 
We call a valuation $\oZ:\fconvf\to\R$ \emph{dually epi-translation invariant} if 
$\oZ^*$ is  epi-translation invariant or equivalently if
$$\oZ(v+\ell+\alpha)=\oZ(v)$$
\goodbreak\noindent
for every $v\in\fconvf$, every linear functional $\ell:\R^n\to\R$ and every $\alpha\in\R$. We say that $\oZ$ is \emph{homogeneous} of degree $j$ if 
$\oZ^*$ is epi-homogeneous of degree $j$ or equivalently if

$$\oZ(\lambda \, v) = \lambda^j \, \oZ(v)$$
for every $v\in\fconvf$ and  $\lambda>0$.
\goodbreak

\section{Cauchy--Kubota Formulas}
\label{se:ck_direct}
In this section, we give a new proof of Theorem~\ref{thm:existence_singular_hessian_vals} and establish the Cauchy--Kubota formulas from Theorem~\ref{thm:cauchy_kubota_for_hessian_vals}. 
In the proofs, we require results on projection functions that we prove in the first part. 
Then we introduce and discuss
the integral transform $\cR$ that connects the coefficient functions in our two versions of the Hadwiger theorem on convex functions.
Finally, we establish Cauchy--Kubota formulas first for smooth functions and then in the general case.

\subsection{Projection Functions}
\label{se:proj_fcts}
For a linear subspace $E\subseteq \R^n$ and a function $u\in\fconvs$, we define the \emph{projection function} $\proj_E u\colon E\to \R$ by
$$\proj_E u(x_E):= \min\nolimits_{z\in E^\perp} u(x_E+z),$$
where $x_E\in E$ and  $E^\perp$ is the orthogonal complement of $E$. Note that this minimum is attained since $u$ is lower semicontinuous and super-coercive. Since $\min_{z\in E^\perp} u(x_E+z) \leq t$ if and only if there exists $z\in E^\perp$ such that $u(x_E+z)\leq t$, this implies that
\begin{equation}
\label{eq:proj_fct_lvl_set}
\{\proj_E u \leq t\} = \proj_E \{u\leq t\}
\end{equation}
for every $t\in\R$ and 
\begin{equation}
\label{eq:proj_fct_epi_grph}
\epi \proj_E u = \proj_{E\times \R} \epi u.
\end{equation}
In particular, it is clear that $\proj_E u \in \fconvsE$.

\begin{lemma}
\label{le:proj_subgradient}
Let $E\subseteq \R^n$ be a linear subspace and let $u\in\fconvs$. If $x_E,y_E\in E$ are such that $y_E\in\partial \proj_E u(x_E)$, then for every $x\in\R^n$ with $\proj_E x = x_E$ and $\proj_E u(x_E)=u(x)$ also $y_E\in\partial u(x)$. In particular, such $x\in\R^n$ exist.
\end{lemma}
\begin{proof}
Let $x_E,y_E$ be given with $y_E\in\partial \proj_E u(x_E)$. By the definition of the projection function, there exists $x\in\R^n$ with $\proj_E x=x_E$ such that
$$\proj_E u(x_E)=\min_{z\in E^\perp} u(x_E+z) =u(x).$$
Since $y_E\in\partial \proj_E u(x_E)$, we have
$$\proj_E u(z_E) \geq \proj_E u(x_E) + \langle z_E-x_E,y_E\rangle$$
for every $z_E\in E$. Thus, using again the definition of the projection function as well as the fact that $\langle w,y_E\rangle = \langle \proj_E w,y_E\rangle$ for every $w\in\R^n$, we obtain
$$u(z) \geq \proj_E u(\proj_E z) \geq u(x) + \langle z-x,y_E\rangle
$$
for every $z\in\R^n$, which shows that $y_E\in\partial u(x)$.
\end{proof}

\goodbreak
Since for every linear subspace $E\subseteq \R^n$ the map $K\mapsto \proj_E K$ is continuous on $\cK^n$, we directly obtain the following result from \eqref{eq:proj_fct_lvl_set} and Lemma~\ref{le:hd_conv_lvl_sets}.

\begin{lemma}
\label{le:projection_function_continuous}
For every linear subspace $E\subseteq \R^n$, the map $\proj_E:\fconvs\to\fconvsE$ is continuous.
\end{lemma}

We also need the next result.
\begin{lemma}
\label{le:joint_cont_son_fconvs}
The map
$$(\vartheta,u)\mapsto u\circ \vartheta^{-1}$$
is jointly continuous on $\SO(n)\times \fconvs$.
\end{lemma}
\begin{proof}
Let $u_l$ be a sequence of functions in $\fconvs$ that epi-converges to some $\bar{u}\in\fconvs$. Furthermore, let $\vartheta_l$ be a convergent sequence in $\SO(n)$ and without loss of generality we may assume that $\vartheta_l x \to x$ for every $x\in\R^n$ as $l\to\infty$. We need to show that $u_l\circ \vartheta_l^{-1}$ epi-converges to $\bar{u}$. This is equivalent, by Lemma~\ref{wijsman}, to the epi-convergence of the corresponding sequence of convex conjugates in $\fconvf$, which on $\fconvf$ is equivalent to pointwise convergence and to uniform convergence on compact sets. Let $v_l, \bar{v}\in\fconvf$ be defined as $v_l:=u_l^*$ for $l\in\N$ and $\bar{v}:=\bar{u}^*$. Since $v_l$ is uniformly convergent to $\bar{v}$ on compact sets, for every $y\in\R^n$, 
\begin{equation*}
\lim_{l\to\infty} v_l(\vartheta_l^t y) = \bar{v}(y),
\end{equation*}
where $\vartheta_l^t$ denotes the transpose of $\vartheta_l$. Thus, $v_l\circ\vartheta_l^t$ is epi-convergent to $\bar v$, and by Lemma \ref{wijsman}, we obtain that $u_{l}\circ \vartheta_l^{-1}$ is epi-convergent to $\bar{u}$.
\end{proof}

Let $1\leq k \leq n-1$ and $E\in\Grass{k}{n}$. There exists a rotation $\vartheta\in \SO(n)$ such that $\{\vartheta x \colon x\in E\}=\R^k$, where we consider both $E$ and $\R^k$ as subspaces of $\R^n$ (note that $\vartheta$ is not unique). Now, for every $u\in\fconvsE$ we have $u\circ \vartheta^{-1}\in\fconvsk$. Note that the restriction of $\vartheta\in\SO(n)$ to $\R^k$ is an element of $\O(k)$ but not necessarily of $\SO(k)$. For an  $\O(k)$ invariant $\oZ:\fconvsk\to\R$, set
$$\oZ(u):=\oZ(u\circ \vartheta^{-1})$$
for $u\in\fconvsE$. Since $\oZ$ is $\O(k)$ invariant, this definition does not depend on the particular choice of $\vartheta\in\SO(n)$ and $\oZ$ is well-defined on $\fconvsE$. 

For $1\leq k\leq n-1$, define the distance of two linear subspaces $E,F\in\Grass{k}{n}$
as the Hausdorff distance of the convex bodies $\Bn\cap E$ and $\Bn\cap F$. This induces a topology on the Grassmannian $\Grass{k}{n}$, which is used in the proof of the following statement.

\begin{lemma}
\label{le:val_int_grass}
Let $1\leq k\leq n-1$. If $\,\oZ\colon\fconvsk\to\R$ is a continuous, epi-translation and $\O(k)$ invariant valuation, then
\begin{equation}
\label{eq:val_int_grass}
u\mapsto \int_{\Grass{k}{n}} \oZ(\proj_E u)\d E
\end{equation}
defines a continuous, epi-translation and $\O(n)$ invariant valuation on $\fconvs$.
\end{lemma}
\begin{proof}
We will first show that
\begin{equation}
\label{eq:e_u_oz}
(E,u)\mapsto \oZ(\proj_E u)
\end{equation}
is  jointly continuous  on $\Grass{k}{n}\times\fconvs$. For this, let $E_l$ be a convergent sequence in $\Grass{k}{n}$ with limit $\bar{E}\in\Grass{k}{n}$, and let $u_l$ be a sequence in $\fconvs$ that epi-converges to some $\bar{u}\in\fconvs$. We need to show that
\begin{equation}
\label{eq:conv_e_u}
\lim_{l\to\infty}\oZ(\proj_{E_l} u_l) = \oZ(\proj_{\bar{E}} \bar{u}).
\end{equation}

Since $E_l$ converges to $\bar{E}$, we may choose a sequence $\vartheta_l\in\SO(n)$ such that $\vartheta_l x \to x$ for every $x\in\R^n$ as $l\to\infty$ and such that $\{\vartheta_l x\colon x\in E_l\}=\bar{E}$ for every $l\in\N$. In particular, we now have
$$(\proj_{E_l} u_l)\circ \vartheta_l^{-1}\in\fconvsbE$$
for every $l\in\N$. By the $\O(k)$ invariance of $\oZ$, the definition of $w\mapsto \oZ(\proj_E w)$ on $\fconvs$ and our choice of $\vartheta_l$, it follows that
$$\oZ(\proj_{E_l} u_l) = \oZ((\proj_{E_l} u_l)\circ \vartheta_l^{-1}) = \oZ(\proj_{\bar{E}} (u_l \circ \vartheta_l^{-1}))$$
for every $l\in\N$. Combined with Lemma~\ref{le:projection_function_continuous} and Lemma~\ref{le:joint_cont_son_fconvs}, this implies \eqref{eq:conv_e_u}.

Next, let $u_l$ be again an epi-convergent sequence in $\fconvs$ with limit $\bar{u}\in\fconvs$. Since $u_l$ is epi-convergent, $\Grass{k}{n}$ is compact, and the map defined by \eqref{eq:e_u_oz} is continuous, the supremum
$$\sup\{\vert \oZ(\proj_E u_l) \vert : l\in\N, E\in\Grass{k}{n}\}$$
is finite. Hence, it follows from the dominated convergence theorem that
$$\lim_{l\to\infty} \int_{\Grass{k}{n}} \oZ(\proj_E u_l) \d E = \int_{\Grass{k}{n}} \oZ(\proj_E \bar{u}) \d E$$
and therefore \eqref{eq:val_int_grass} is continuous. In particular, the right side of \eqref{eq:val_int_grass} is well-defined and finite. In addition, it is easy to see that \eqref{eq:val_int_grass} is epi-translation and $\O(n)$ invariant. Finally, the valuation property follows from the corresponding property of $\oZ$ combined with the fact that
$$\proj_E (u\vee v) = (\proj_E u) \vee (\proj_E v),\quad\quad \proj_E(u\wedge v) = (\proj_E u) \wedge (\proj_E v)$$
for every $u,v\in\fconvs$ and $E\in \Grass{k}{n}$.
\end{proof}

\goodbreak

As a consequence of
Proposition~\ref{prop:n-hom} and Lemma~\ref{le:val_int_grass}, we obtain the following result.

\begin{lemma}
\label{le:cauchy_kubota_is_continuous_val}
For $0\leq j \leq n$ and $\alpha\in C_c([0,\infty))$, the functional
\begin{equation*}
u\mapsto \int_{\Grass{j}{n}} \int_{\dom (\proj_E u)} \alpha(|\nabla \proj_E u(x_E)|) \d x_E \d E
\end{equation*}
is a continuous, epi-translation and rotation invariant valuation on $\fconvs$.
\end{lemma}

\subsection{The Integral Transform $\cR$}
For $\zeta\in C_b((0,\infty))$ and $s>0$, define
\begin{equation*}
\cR \zeta(s):= s\, \zeta(s) + \int_s^{\infty}  \zeta(t)\d t.
\end{equation*}
Note that, under these assumptions, we have $\cR \zeta \in C_b((0,\infty))$. For $l\in \N$, let 
$$\cR^l \zeta:=\underbrace{(\cR \circ \cdots \circ \cR)}_{l} \zeta$$ 
and set $\cR^0 \zeta := \zeta$.

\begin{lemma}\label{Rtol}
If $\,l\ge 0$ and $\zeta\in C_b((0,\infty))$, then
$$\cR^l \zeta(s) = s^l \zeta(s) + l \int_s^{\infty} t^{l-1} \zeta(t) \d t$$
for $s>0$.
\end{lemma}
\begin{proof}
We prove the statement by induction on $l$. Observe that the statement is trivially true for $l=0$ and $l=1$. Therefore, assume that $l>1$ and that the statement is true for the case $l-1$. Using the induction assumption, we now have
\begin{align}
\begin{split}
\label{eq:calc_cR_l}
\cR^l \zeta(s) &= \cR^{l-1} \cR \zeta(s)\\
&= s^l \zeta(s) + s^{l-1} \int_s^{\infty} \zeta(t) \d t + (l-1) \int_s^{\infty} t^{l-1} \zeta(t) \d t + (l-1) \int_s^{\infty} t^{l-2} \int_t^{\infty}  \zeta(r) \d r \d t
\end{split}
\end{align}
for every $s>0$. Using integration by parts and  that $\zeta$ has bounded support shows that
$$(l-1) \int_s^{\infty} t^{l-2} \int_t^{\infty} \zeta(r) \d r \d t = -s^{l-1} \int_s^{\infty} \zeta(t) \d t + \int_s^{\infty} t^{l-1} \zeta(t) \d t$$
for every $s>0$, which combined with \eqref{eq:calc_cR_l} completes the proof.
\end{proof}

\goodbreak
We require the following simple result.

\begin{lemma}
\label{le:lim_t_int_zeta_new}
    Let $0\leq k < n-1$. If $\zeta\in\Had{k}{n}$, then
    \begin{equation}
    \label{eq:lim_t_int_zeta_new_1}
        \lim_{s\to 0^+} s^{n-1-k} \int_s^{\infty} \zeta(t) \d t = 0.
    \end{equation}
    Moreover, if $\,0\leq k<n$ and $\rho\in\Had{k}{n-1}$, then
    $$\lim_{s\to 0^+} s^{n-k} \int_s^{\infty} \frac{\rho(t)}{t^{2}} \d t = 
	\begin{cases}\rho(0) \quad &\text{if } \,k=n-1,\\[2pt]0 \quad &\text{else.}\end{cases}$$
\end{lemma}
\begin{proof}
Let $\zeta\in\Had{k}{n}$. If $\zeta$ is such that $\lim_{s\to 0^+}\int_s^{\infty} \zeta(t)\d t$ exists and is finite, then \eqref{eq:lim_t_int_zeta_new_1} is trivial. In the remaining case, we use L'Hospital's rule and the definition of $\Had{k}{n}$ to obtain
$$
\lim_{s\to 0^+} \Big| s^{n-1-k} \int_s^{\infty} \zeta(t) \d t\Big| \leq \lim_{s\to 0^+} s^{n-1-k} \int_s^{\infty} |\zeta(t)|\d t= \lim_{s\to 0^+} \frac{|\zeta(s)|}{\frac{n-1-k}{s^{n-k}}} = \lim_{s\to 0^+} \frac{|s^{n-k}\zeta(s)|}{n-1-k}=0.
$$
The proof of the second statement is analogous. We remark that for $k=n-1$ the limit $\lim_{s\to 0^+} \rho(s)=\rho(0)$ exists and is finite.
\end{proof}

\goodbreak
In the following lemma, basic properties of the integral transform $\cR$ are established.

\begin{lemma}\label{le:r_kln}
For $0\leq k \leq n$ and $0\leq l \leq n-k$, the map $\cR^l:\Had{k}{n}\to\Had{k}{n-l}$ is a bijection with inverse $\cR^{-l}\colon \Had{k}{n-l}\to \Had{k}{n}$, given by
\begin{equation}
    \label{eq:cR_inv}
	\cR^{-l} \rho(s):=(\cR^{-1})^l \rho(s)=\frac{\rho(s)}{s^{l}} -l \int_s^{\infty} \frac{\rho(t)}{t^{l+1}} \d t
\end{equation}
for $\rho\in \Had{k}{n-l}$ and $s>0$.
\end{lemma}
\begin{proof}
    Let $0\leq k \leq n-1$. We will first show that if $\zeta\in\Had{k}{n}$, then $\cR \zeta\in\Had{k}{n-1}$. In case $k=n-1$, it easily follows from the definition of $\Had{n-1}{n}$ that $\lim_{s\to 0^+}\cR\zeta(s)$ exists and is finite and thus $\cR \zeta\in \Had{n-1}{n-1}$. In case $k<n-1$, we have
    $$s^{n-1-k} \cR \zeta(s) = s^{n-k} \zeta(s)+s^{n-1-k}\int_s^{\infty} \zeta(t) \d t$$
    for $s>0$. Since $\zeta\in\Had{k}{n}$, it follows that $\lim_{s\to 0^+} s^{n-k} \zeta(s)=0$. Combined with Lemma~\ref{le:lim_t_int_zeta_new} this shows that 
	$$\lim_{s\to 0^+}s^{n-1-k} \cR \zeta(s)=0.$$
	Next, observe that
    \begin{align*}
	\int_s^{\infty} t^{n-1-k-1} \cR \zeta(t)  \d t &=  \int_s^{\infty} t^{n-k-1}\zeta(t)  \d t+ \int_{s}^{\infty} t^{n-1-k-1} \int_t^{\infty} \zeta(r)  \d r \d t\\
	&=  \int_s^{\infty} t^{n-k-1}\zeta(t)  \d t
	-\frac{s^{n-1-k}}{n-1-k} \int_s^{\infty}\zeta(t)\d t + \int_s^{\infty} \frac{t^{n-1-k}}{n-1-k}\zeta(t)  \d t\\
	&=  \frac{n-k}{n-1-k}\int_s^{\infty} t^{n-k-1} \zeta(t)  \d t -\frac{1 }{n-1-k}\,s^{n-1-k} \int_s^{\infty} \zeta(t)\d t .
	\end{align*}
	Since $\zeta\in\Had{k}{n}$, we see that $\lim_{s\to 0^+} \int_s^{\infty} t^{n-k-1}\zeta(t) \d t$ exists and is finite. Combined with Lemma~\ref{le:lim_t_int_zeta_new}, this shows that the expression above converges to a finite value as $s\to 0^+$. Thus, $\cR \zeta\in\Had{k}{n-1}$. It now easily follows by induction that $\cR^l \zeta \in \Had{k}{n-l}$ for $0\leq k \leq n$ and $0\leq l \leq n-k$, where we remark that the case $l=0$ is trivial.
	
	Second, for $\zeta\in\Had{k}{n}$ we have
	$$\frac{\cR^l \zeta(s)}{s^l}-l\int_s^{\infty}\frac{\cR^l \zeta(t)}{t^{l+1}} \d t = \zeta(s) + \frac{l}{s^l}\int_s^{\infty} t^{l-1} \zeta(t)\d t - l \int_s^{\infty} \frac{\zeta(t)}{t} \d t - l^2 \int_s^{\infty} \frac{1}{t^{l+1}}\int_t^{\infty} r^{l-1} \zeta(r) \d r \d t$$
	for every $s>0$. Using integration by parts, we obtain
	$$\int_s^{\infty} \frac{1}{t^{l+1}} \int_t^{\infty} r^{l-1} \zeta(r) \d r \d t = \frac{1}{l} \Big(\frac{1}{s^l} \int_s^{\infty} t^{l-1} \zeta(t) \d t - \int_s^{\infty} \frac{\zeta(t)}{t} \d t\Big)$$
	for $s>0$ and therefore the (left) inverse of $\cR^l$ is given by \eqref{eq:cR_inv}. Similarly, one shows that $\cR^l$ is the inverse operation to \eqref{eq:cR_inv}.
	
	Now let $\rho\in\Had{k}{n-1}$ with $0\leq k \leq n-1$ be given. We need to show that $\cR^{-1} \rho \in \Had{k}{n}$. Again, it is easy to see that the continuity and the bounded support of $\rho$ imply the same properties for $\cR^{-1}\rho$. Since
	$$
	s^{n-k}\cR^{-1} \rho(s)  = s^{n-k-1}\rho(s) - s^{n-k} \int_s^{\infty} \frac{\rho(t)}{t^2} \d t,
	$$
    it  follows from the definition of $\Had{k}{n-1}$ and Lemma~\ref{le:lim_t_int_zeta_new} that $\lim_{s\to 0^+} s^{n-k} \cR^{-1}  \rho(s)=0$. Note, that in the last step the cases $k<n-1$ and $k=n-1$ need to be dealt with separately. Furthermore, observe that
    \begin{align*}
	\int_s^{\infty} t^{n-k-1}\cR^{-1} \rho(t) \d t &= \int_s^{\infty}  t^{n-1-k-1} \rho(t) \d t -\int_s^{\infty} t^{n-k-1} \int_t^{\infty} \frac{\rho(r)}{r^{2}}\d r \d t \\
	&= \int_s^{\infty} t^{n-1-k-1}\rho(t) \d t +\frac{s^{n-k}}{n-k} \int_s^{\infty} \frac{\rho(t)}{t^{2}}\d s - \int_s^{\infty} \frac{t^{n-k}}{n-k} \frac{\rho(t)}{t^{2}} \d t \\
	&= \frac{n-1-k}{n-k} \int_s^{\infty} t^{n-1-k-1} \rho(t)  \d t + \frac{1}{n-k} s^{n-k} \int_s^{\infty} \frac{\rho(t)}{t^{2}} \d t.
	\end{align*}
	In case $k=n-1$, the first term on the right side of the last equation vanishes. In case $k<n-1$, it follows from the definition of $\Had{k}{n-1}$ that $\lim_{s\to 0^+} \int_s^{\infty} t^{n-1-k-1} \rho(t) \d t$ exists and is finite and from Lemma~\ref{le:lim_t_int_zeta_new} that the second term converges as $s\to0^+$. Thus, $\cR^{-1}\rho \in \Had{k}{n}$ and $\cR:\Had{k}{n}\to\Had{k}{n-1}$ is a bijection.

    Finally, it now easily follows by induction that $\cR^l:\Had{k}{n}\to \Had{k}{n-l}$ is a bijection for $0\leq k \leq n$ and $0\leq l \leq n-k$, where again the case $l=0$ is trivial. Furthermore, this implies that $(\cR^{-1})^l$ is indeed given by \eqref{eq:cR_inv}.
\end{proof}

We remark that since $\Had{k}{k}=\Had{n}{n}$ for every $0\leq k <n$, Lemma~\ref{le:r_kln} allows us to redefine $\Had{k}{n}$ as
$$\Had{k}{n}=\cR^{-(n-k)} \Had{n}{n}=\{ \cR^{-(n-k)} \zeta \colon \zeta \in \Had{n}{n}\}.$$

\subsection{Cauchy--Kubota Formulas for Smooth Functions} 
We use the auxiliary space,
\begin{align*}
\fconvso=\{u\in\fconvs\colon u(0)=0\leq u(x)\text{ for every } x\in\R^n\}.
\end{align*}
If $u\in\fconvso\cap C_+^2(\R^n)$, then the level sets $\{u\leq t\}$ have boundary of class $C^2$ with positive Gaussian curvature for every $t>0$. We have $u(x)=0$ if and only if $x=0$. For such a function $u$ and $0\leq j \leq n-1$, we write $\tau_j(u,x)$ for the $j$th elementary symmetric function of the principal curvatures of $\{u\leq t\}$ at $x\ne0$, where $t=u(x)$.

\goodbreak
We need the following result, whose proof is based on a lemma by Reilly \cite{Reilly}.

\begin{proposition}[\!\cite{Colesanti-Ludwig-Mussnig-5}, Proposition 3.13]
\label{prop:hugo1_313}
Let $1\leq j \leq n-1$ and $\zeta\in\Had{j}{n}$. For $0<t_1<t_2$ and $u\in\fconvso\cap C_+^2(\R^n)$,
\begin{align*}
\int_{\{t_1<u\leq t_2\}} \zeta(|\nabla u(x)|) \big[\Hess u(x)\big]_{n-j} \d x = &\int_{\{t_1 < u \leq t_2\}} (\cR^{n-j}\zeta)(|\nabla u(x)|)\, \tau_{n-j}(u,x) \d x\\
&- \int_{\{u=t_2\}} \eta_{n-j-1}(|\nabla u(x)|)\, \tau_{n-j-1}(u,x) \d \mathcal{H}^{n-1}(x)\\
&+ \int_{\{u=t_1\}} \eta_{n-j-1}(|\nabla u(x)|)\, \tau_{n-j-1}(u,x) \d \mathcal{H}^{n-1}(x),
\end{align*}
where $\eta_{n-j-1}(s)=\int_s^{\infty} t^{n-j-1} \zeta(t) \d t$ for $s>0$.
\end{proposition}

\goodbreak
As a consequence, we obtain the following lemma.

\begin{lemma}\label{needed lemma} Let $1\leq j\leq n-1$ and $\zeta\in\Had{j}{n}$. For $u\in\fconvso\cap C_+^2(\R^n)$,
$$
\int_{\R^n} \zeta(|\nabla u(x)|) \big[\Hess u(x)\big]_{n-j} \d x= \int_{\R^n} (\cR^{n-j}\zeta)(|\nabla u(x)|)\, \tau_{n-j}(u,x) \d x.
$$
\end{lemma}

\begin{proof}
Since $u(0)\geq u(x)+\langle \nabla u(x),-x\rangle$ it follows from the Cauchy--Schwarz inequality that
$$|\nabla u(x)| \geq \frac{u(x)-u(0)}{|x|}$$
for every $x\in\R^n\backslash\{0\}$. Using that
$\lim_{|x|\to\infty} u(x)/|x|=+\infty$, we obtain  that $\lim_{|x|\to\infty}|\nabla u(x)|=+\infty$. The proof now follows by letting $t_1\to0^+$ and $t_2\to\infty$ in Proposition \ref{prop:hugo1_313}.
Here, for the integral involving $t_1$ we use that $\eta_{n-j-1}$ is bounded and thus, since $\{u=0\}=\{0\}$ and because of \eqref{curvature measures smooth case}, this integral vanishes as $t_1\to 0^+$. For the integral involving $t_2$, we use the fact that $\eta_{n-j-1}$ has compact support.
\end{proof}

We can now prove Cauchy--Kubota formulas for convex functions in $C_+^2(\R^n)$.

\begin{proposition}
\label{prop:ck_regular}
Let $1\leq j \leq k <n$. If $\zeta\in\Had{j}{n}$, then
\begin{align}
\label{eq:ck_regular}
\int_{\R^n} &\zeta(|\nabla u(x)|) \big[\Hess u(x)\big]_{n-j} \d x \nonumber\\
\\[-16pt]
&= \frac{\kappa_n}{\kappa_k \kappa_{n-k} }\binom{n}{k}
\int_{\Grass{k}{n}} \int_E \xi(|\nabla \proj_E u(x_E)|) \,\big[\Hess \proj_E u(x_E)\big]_{k-j} \d x_E \d E
\nonumber
\end{align}
for every $u\in\fconvs\cap C_+^2(\R^n)$, where $\xi\in\Had{j}{k}$ is given by
$$\xi(s):= \frac{\kappa_{n-k}}{\binom{n-j}{k-j}} \cR^{n-k} \zeta(s)$$
for $s>0$.
\end{proposition}
\begin{proof}
Let $K\in\cK^n$ be of class $C^{2}$ with positive Gaussian curvature. In particular, this implies that $K$ is strictly convex. For $E\in\Grass{k}{n}$, let $\bd_E \proj_E K$ denote the boundary of $\proj_E K$ as a subset of $E$. It follows from the strict convexity of $K$ that for every $x_E\in\bd_E\proj_E K$, there exists a unique point $x\in\bd K$ such that $\proj_E x=x_E$. The map $x_E\mapsto x$ can be also defined as follows. Let $\nu_K\colon\partial K\to\sn$ be the Gauss map of $K$, and let $\nu_{\proj_E K}\colon\bd_E\proj_E K\to\mathbb{S}^{k-1}_E$ be the Gauss map of $\proj_E K$ (here the unit sphere $\mathbb{S}^{k-1}_E$ of $E$ is seen as a subset of $\sn$). Then $\nu_K$ and $\nu_{\proj_E K}$ are diffeomorphisms and 
$$
x=\nu_K^{-1}(\nu_{\proj_E K}(x_E)).
$$
For simplicity, we write $x=\proj_E^{-1} x_E$, that is, $\proj_E^{-1}=\nu_K^{-1}\circ\nu_{\proj_E K}$. Let $\gamma\colon\bd K\to\R$ be continuous (which implies in particular that $\gamma\circ\proj_E^{-1}$ is continuous). It follows from \eqref{curvature measures smooth case}, \eqref{Schneider1} combined with Fubini's theorem, and again \eqref{curvature measures smooth case} (in dimension $k$) that
\begin{align}
\begin{split}
\label{curv_proj}
\int_{\bd K} &\gamma(x) \,\tau_{n-j}(K,x)\d \hm^{n-1}(x)\\
&= \binom{n-1}{n-j} \int_{\bd K} \gamma(x) \d C_{j-1}(K,x)\\
&= \binom{n-1}{n-j} \frac{n \kappa_n}{k \kappa_k} \int_{\Grass{k}{n}} \int_{\bd_E \proj_E K} \gamma(\proj_E^{-1} x_E) \d C_{j-1}^E(\proj_E K,x_E) \d E\\
&=
\frac{\kappa_n \binom{n}{j}}{\kappa_k\binom{k}{j}}\int_{\Grass{k}{n}}\int_{\bd_E \proj_E K}\gamma(\proj_E^{-1}x_E)\, \tau^E_{k-j}(\proj_E K, x_E)\d \hm^{k-1}(x_E)\d E.
\end{split}
\end{align}
Here, $\tau_{k-j}^E(\proj_E K, x_E)$ is the $(k-j)$th elementary symmetric function of the principal curvatures of $\bd_E \proj_E K$ at $x_E$ in $E$.

Now, let $u\in\fconvso\cap C_+^2(\R^n)$ and $0<t_1<t_2$. We first observe that, by the coarea formula, 
\begin{equation*}
 \int_{\{t_1<u\le t_2\}} (\cR^{n-j}\zeta)(|\nabla u(x)|)\, \tau_{n-j}(u,x) \d x
=\int_{t_1}^{t_2}\int_{\{u=t\}}\frac{(\cR^{n-j} \zeta)(|\nabla u(x)|)}{|\nabla u(x)|}\tau_{n-j}(u,x)\d\mathcal{H}^{n-1}(x)\d t.
\end{equation*}

Next, fix $E\in \Grass{k}{n}$. For every $t>0$, the convex set $\{u\leq t\}$ has positive Gaussian curvature. We consider the map $\proj_E^{-1}\colon\bd_E\proj_E\{u\leq t\}\to\{u= t\}$ defined as above. Combined with Lemma~\ref{le:proj_subgradient} we therefore have
\begin{equation}\label{proj_grad}
\nabla u(\proj_E^{-1}x_E)=\nabla \proj_E u(x_E)
\end{equation}
for every $x_E\in\bd_E\proj_E\{u\le t\}$.

Hence
\begin{align*}
&\int_{\{t_1<u\le t_2\}} (\cR^{n-j}\zeta)(|\nabla u(x)|) \,\tau_{n-j}(u,x) \d x\\
&=\int_{t_1}^{t_2}\int_{\{u=t\}}\frac{(\cR^{n-j} \zeta)(|\nabla u(x)|)}{|\nabla u(x)|} \tau_{n-j}(u, x) \d \hm^{n-1}(x)\d t\\
&=\frac{\kappa_n \binom{n}{j}}{\kappa_k\binom{k}{j}}\int_{t_1}^{t_2}\int_{\Grass{k}{n}}\int_{\bd_E \proj_E \{u\le t\}}
\frac{(\cR^{n-j} \zeta)(|\nabla u(\proj_E^{-1}x_E)|)}{|\nabla u(\proj_E^{-1}x_E)|}
\tau^E_{k-j}(\proj_E u,x_E)\d \hm^{k-1}(x_E)\d E\d t\\
&=\frac{\kappa_n \binom{n}{j}}{\kappa_k\binom{k}{j}}\int_{\Grass{k}{n}}\int_{t_1}^{t_2}\int_{ \{\proj_E u = t\}}
\frac{(\cR^{n-j} \zeta)(|\nabla\proj_E u(x_E)|)}{|\nabla\proj_E u(x)|}\,
\tau^E_{k-j}(\proj_E u, x_E) \d \mathcal{H}^{k-1}(x_E)\d t \d E\\
&=\frac{\kappa_n \binom{n}{j}}{\kappa_k\binom{k}{j}}\int_{\Grass{k}{n}}\int_{\{t_1<\proj_E u\le t_2\}} (\cR^{n-j}\zeta)(|\nabla \proj_E u(x_E)|)\,\tau^E_{k-j}(\proj_E u, x_E)\d x_E\d E,
\end{align*}
where we have used the coarea formula, (\ref{curv_proj}), (\ref{proj_grad}), and Fubini's theorem.
Next, let $t_1\to 0^+$ and $t_2\to +\infty$. We apply Lemma~\ref{needed lemma} to both $u$ and to $\proj_E u$. On the right side we also use the boundedness of $\cR^{n-j}\zeta$,  equation \eqref{curvature measures smooth case} and the dominated convergence theorem, and obtain
\begin{multline*}
\int_{\R^n} \zeta(|\nabla u(x)|)\,\big[\Hess u(x)\big]_{n-j} \d x\\
= \frac{\kappa_n \binom{n}{j}}{\kappa_k\binom{k}{j}} \int_{\Grass{k}{n}} \int_E \cR^{-(k-j)} (\cR^{n-j} \zeta)(|\nabla \proj_E u(x_E)|)[\Hess \proj_E u(x_E)]_{k-j} \d x_E \d _E.
\end{multline*}
Since $\cR^{-(k-j)} \cR^{n-j} \zeta = \cR^{n-k} \zeta$ and
$$\binom{n}{j}\binom{n-j}{k-j}=\binom{n}{k}\binom{k}{j},$$
we have therefore shown \eqref{eq:ck_regular} for $u\in\fconvso\cap C_+^2(\R^n)$. The conclusion now follows since for each $u\in\fconvs\cap C_+^2(\R^n)$ there exists $u_0\in\fconvso\cap C_+^2(\R^n)$ such that $\epi(u_0)$ is a translate of $\epi(u)$ in $\R^{n+1}$ and since both sides of \eqref{eq:ck_regular} are invariant with respect to epi-translations.
\end{proof}

\goodbreak
For the special case $j=k$, we immediately obtain the following result, where we use that each function in $\Had{j}{j}$ can be uniquely extended to a function in $C_c([0,\infty))$.

\begin{proposition}
\label{prop:ck}
Let $1\leq j <n$. If $\zeta\in\Had{j}{n}$, then 
\begin{equation*}
\int_{\R^n} \zeta(|\nabla u(x)|) \big[\Hess u(x)\big]_{n-j} \d x=\frac{\kappa_n}{\kappa_j \kappa_{n-j}} \binom{n}{j} \int_{\Grass{j}{n}} \int_{\dom (\proj_E u)} \alpha(|\nabla \proj_E u(x_E)|) \d x_E \d E
\end{equation*}
for every $u\in\fconvs\cap C_+^2(\R^n)$, where $\alpha\in C_c([0,\infty))$ is given by
$$\alpha(s):=  \kappa_{n-j} \cR^{n-j} \zeta(s)$$ 
for $s>0$. 
\end{proposition}

\goodbreak
\subsection{New Proof of Theorem~\ref{thm:existence_singular_hessian_vals}}
The case $j=n$ follows from
Proposition~\ref{prop:n-hom} and the case $j=0$ is trivial. So let $1\leq j \leq n-1$. For $\zeta\in\Had{j}{n}$, define
$$\alpha(s):= \kappa_{n-j}\cR^{n-j}\zeta(s)$$
for $s>0$ and note that $\alpha$ can be extended to a function in $C_c([0,\infty))$ by Lemma \ref{le:r_kln} and the definition of $\Had{j}{j}$.
Hence Lemma~\ref{le:cauchy_kubota_is_continuous_val} shows that the functional $\oZ$, defined by
$$\oZ(u):= \frac{\kappa_n}{\kappa_j \kappa_{n-j}} \binom{n}{j} \int_{\Grass{j}{n}} \int_{\dom (\proj_E u)} \alpha(|\nabla \proj_E u(x_E)|) \d x_E \d E,$$
is a continuous, epi-translation and rotation invariant valuation on $\fconvs$. 
From Proposition~\ref{prop:ck}, we obtain that 
\begin{equation}\label{eq:new_proof_z_equal_regular}
\oZ(u)=  \int_{\R^n} \zeta(|\nabla u(x)|) \big[\Hess u(x)\big]_{n-j} \d x
\end{equation}
for $u\in\fconvs\cap C_+^2(\R^n)$. Thus $\oZ$ has the required properties. It is uniquely determined by \eqref{eq:new_proof_z_equal_regular} since  $\fconvs\cap C_+^2(\R^n)$ is dense in $\fconvs$ by Lemma~\ref{le:C2p_dense}.

\subsection{An Auxiliary Result} 
Using polar coordinates, we obtain from \eqref{eq:rep_ozz_c2} and \eqref{eq:int_u_psi_int_v_phi}  that
\begin{equation}
\label{eq:v_0_zeta}
\oZZ{0}{\zeta}^n(u)=\int_{\R^n} \zeta(|x|) \d x = n \, \kappa_n \lim_{s\to 0^+} \int_s^{\infty} t^{n-1} \zeta(t) \d t
\end{equation}
for every $u\in\fconvs$ and $\zeta\in\Had{0}{n}$.
The following result proves the case $j=0$ in Theorem~\ref{thm:cauchy_kubota_for_hessian_vals}.

\begin{lemma}
\label{le:ck_zero}
Let $0\leq k < n$. If $\zeta\in\Had{0}{n}$, then
$$\oZZ{0}{\zeta}^n(u) = \frac{\kappa_n}{\kappa_k} \int_{\Grass{k}{n}} \oZZ{0}{\cR^{n-k}\zeta}^k(\proj_E u) \d E$$
for every $u\in\fconvs$.
\end{lemma}
\begin{proof}
Observe that by \eqref{eq:v_0_zeta}, Lemma~\ref{le:r_kln} and the definition of $\Had{0}{n}$, we have
$$\oZZ{0}{\zeta}^n(u)=n\, \kappa_n \lim_{s\to 0^+} \int_s^{\infty} t^{n-1}\zeta(t) \d t  = \kappa_n \cR^{n}\zeta(0)$$
for every $u\in\fconvs$ and similarly
$$\oZZ{0}{\cR^{n-k}\zeta}^k(\proj_E u) = \kappa_k \cR^k \cR^{n-k}\zeta(0) = \kappa_k \cR^n \zeta(0)$$
for every $u\in\fconvs$ and $E\in\Grass{k}{n}$. Combined with our conventions for the case $k=0$, the statement is now immediate.
\end{proof}

\subsection{Proof of Theorem~\ref{thm:cauchy_kubota_for_hessian_vals}}
The case $j=0$ follows from Lemma~\ref{le:r_kln} and Lemma~\ref{le:ck_zero}. Therefore, assume that $j>0$. For $u\in\fconvs\cap C_+^2(\R^n)$, it follows from Theorem~\ref{thm:existence_singular_hessian_vals} and Proposition~\ref{prop:ck_regular} that
\begin{align*}
\oZZ{j}{\zeta}^n(u)&=\int_{\R^n} \zeta(|\nabla u(x)|)\big[\Hess u(x)\big]_{n-j}\d x\\
&= 
\frac{\kappa_n}{\kappa_k \kappa_{n-k} }\binom{n}{k}
\int_{\Grass{k}{n}}\int_E \xi(|\nabla \proj_E u(x_E)|) \big[\Hess \proj_E u(x_E)\big]_{k-j} \d x_E \d E\\
&=\frac{\kappa_n}{\kappa_k \kappa_{n-k} }\binom{n}{k}
\int_{\Grass{k}{n}} \oZZ{j}{\xi}^k(\proj_E u) \d E,
\end{align*}
where $\xi\in\Had{j}{k}$ is as in \eqref{eq:def_xi}. The statement now follows from Theorem~\ref{thm:existence_singular_hessian_vals}, Lemma~\ref{le:val_int_grass} and Lemma~\ref{le:C2p_dense}.

\section{The Hadwiger Theorem on Finite-Valued Convex Functions}
\label{se:dual}
The authors \cite{Colesanti-Ludwig-Mussnig-5} established the Hadwiger theorem also for valuations on $\fconvf$ by using duality with valuations on $\fconvs$. For $0\le j\le n$ and $\zeta\in\Had{j}{n}$, define $\oZZ{j}{\zeta}^{n,*}$ as the valuation dual to $\oZZ{j}{\zeta}^{n}$, that is, $\oZZ{j}{\zeta}^{n,*}(v):=\oZZ{j}{\zeta}^{n}(v^*)$ for $v\in\fconvf$. 

\begin{theorem}[\!\cite{Colesanti-Ludwig-Mussnig-5}, Theorem 1.4]\label{dual main one way} 
For $\,0\leq j \leq n$ and $\zeta\in\Had{j}{n}$, 
the functional
$\oZZ{j}{\zeta}^{n,*}\colon\fconvf\to\R$ is a continuous, dually epi-translation and 	rotation invariant valuation  such that
\begin{equation*}
	\oZZ{j}{\zeta}^{n,*}(v)=  \int_{\R^n} \zeta(|x|) \big[\Hess v(x)\big]_{j} \d x
\end{equation*}
for every $v\in\fconvf\cap C^2_+(\R^n)$. 
\end{theorem}

\goodbreak
The Hadwiger theorem on $\fconvf$ is the following result. Let $n\ge2$.

\begin{theorem}[\!\cite{Colesanti-Ludwig-Mussnig-5}, Theorem 1.5]
\label{dthm:hadwiger_convex_functions}
A functional $\oZ:\fconvf \to \R$ is a continuous, dually epi-translation and rotation invariant valuation if and only if there exist  functions $\zeta_0\in\Had{0}{n}$, \dots, $\zeta_n\in\Had{n}{n}$  such that
\begin{equation*}
\oZ(v)= \sum_{j=0}^n \,\oZZ{j}{\zeta_j}^{n,*}(v) 
\end{equation*}
for every $v\in\fconvf$.
\end{theorem}

For $v\in\fconvf$ and a linear subspace $E$ of $\R^n$,  let $v|_E:E\to\R$ denote the restriction of $v$ to $E$. We require the following result. 

\begin{lemma}[\!\cite{RockafellarWets}, Theorem 11.23]
\label{le:conj_proj}
If $E$ is a linear subspace of $\,\R^n$ and $u\in\fconvs$, then
$$(\proj_E u)^*(x_E)=(u^*)|_{E}(x_E)$$
for $x_E\in E$, where on the left side the convex conjugate is taken with respect to the ambient space $E$.
\end{lemma}

\noindent
The following result is obtained from Theorem \ref{thm2:hadwiger_convex_functions_ck} by using Lemma \ref{le:conj_proj}, (\ref{n-gradient}) and (\ref{eq:int_u_psi_int_v_phi}). It is our second version of the Hadwiger theorem on $\fconvf$. Let $n\geq 2$.

\begin{theorem}\label{hugo_ck_f}
A functional $\,\oZ:\fconvf\to\R$ is a continuous, dually epi-translation and rotation invariant valuation if and only if there exist functions $\alpha_0,\ldots,\alpha_n\in C_c([0,\infty))$ such that
$$\oZ(v)=\sum_{j=0}^n \int_{\Grass{j}{n}} \int_{E} \alpha_j(|x|) \d \Phi_{j}^j(v|_E,x) \d E$$
for every $v\in\fconvf$.
\end{theorem}

\noindent
Here in the summand $j=0$ we define $\d\Phi_0^0(v|_E,\cdot)$ to be the Dirac point measure at $0$ and note that this summand is just a constant functional on $\fconvf$.

\goodbreak
The following integral-geometric formulas  are obtained from Theorem \ref{thm:cauchy_kubota_for_hessian_vals} by using Lemma \ref{le:conj_proj}.

\begin{theorem}
\label{thm:cauchy_kubota_for_hessian_vals_dual}
For $\,0\leq j \leq k < n$ and $\zeta\in\Had{k}{n}$,
$$\oZZ{j}{\zeta}^{n,*}(v)=\frac{\kappa_n}{\kappa_k \kappa_{n-k} }\binom{n}{k}
\int_{\Grass{k}{n}} \oZZ{j}{\xi}^{k,*}(v|_E)\d E$$
for every $v\in\fconvf$, where $\xi\in\Had{j}{k}$ is given by
$$\xi(s):= \frac{\kappa_{n-k}}{\binom{n-j}{k-j}} 
\big(
s^{n-k}\zeta(s) +(n-k) \int_s^{\infty}  t^{n-k-1}\zeta(t)\d t\big)$$
for $s>0$.
\end{theorem}

\noindent
For results of a similar nature we refer to \cite[Theorem 2.1]{ColesantiHug2000a}, where Crofton formulas for Hessian measures were established.
The following special case of the previous theorem corresponds to Theorem \ref{cauchy_function}. Combined with properties of the integral transform mapping $\zeta$ to $\alpha$ (see Lemma \ref{le:r_kln}), it shows that Theorem \ref{hugo_ck_f} is equivalent to Theorem \ref{dthm:hadwiger_convex_functions}. 

\begin{theorem}
For $\,0\leq j < n$ and $\zeta\in \Had{j}{n}$,
$$\oZZ{j}{\zeta}^{n,*}(v)=
\frac{\kappa_n}{\kappa_j\kappa_{n-j}} \binom{n}{j}
\int_{\Grass{j}{n}} \int_{E} \alpha(|x|) \d \Phi_{j}^j(v|_E,x) \d E$$
for every $v\in\fconvf$, where $\alpha\in C_c([0,\infty))$ is given by 
$$\alpha(s):= \kappa_{n-j}
\big(  s^{n-j}\zeta(s)+(n-j)\int_s^{\infty}  t^{n-j-1} \zeta(t)\d t\big)$$
for $s>0$.
\end{theorem}

\section{Additional Results and Applications}
\label{se:applications}
In this section, we present a second proof of Theorem~\ref{thm:cauchy_kubota_for_hessian_vals}, which uses Theorem~\ref{thm:hadwiger_convex_functions}, and  establish connections between functional intrinsic volumes and their classical counterparts. We also answer questions about non-negative and monotone valuations.

\goodbreak
We require  the following result, which follows from \cite[Lemma 2.15 and Lemma 3.24]{Colesanti-Ludwig-Mussnig-5}. 

\begin{lemma}[\!\cite{Colesanti-Ludwig-Mussnig-5}]
\label{le:calc_ind_bn_tx_theta_i}
	If $\,1\leq j \leq n$ and  $\zeta\in\Had{j}{n}$, then  
	$$\oZZ{j}{\zeta}^n(u_t)=\kappa_n \binom{n}{j} \cR^{n-j} \zeta(t)$$
	for $t\geq 0$, where $u_t(x):= t \vert x\vert + \ind_{\Bn}(x)$ for $x\in\R^n$.
\end{lemma}

\subsection{Second Proof of Theorem~\ref{thm:cauchy_kubota_for_hessian_vals}}
\label{se:ck_hadwiger}

By Lemma \ref{Rtol}, we have
\begin{equation*}
\xi=\frac{\kappa_n \binom{n}{j}}{\kappa_k\binom{k}{j} }\,\cR^{n-k} \zeta
\end{equation*}
and  Lemma \ref{le:r_kln} implies that
$\xi\in\Had{j}{k}$. For $j=0$, the result now follows from Lemma~\ref{le:ck_zero}. 
Thus, let $j>0$. For every $E\in \Grass{k}{n}$,  it easily follows from \eqref{eq:proj_fct_epi_grph} that $\proj_E (t\sq u) =t \sq \proj_E u$ for every $t>0$ and $u\in\fconvs$. Hence,  using Lemma~\ref{le:val_int_grass}, we obtain that the right side of \eqref{eq:cauchy_kubota_via_hadwiger} defines a continuous, epi-translation and rotation invariant valuation that is epi-homogeneous of degree $j$. Thus, by Theorem~\ref{thm:hadwiger_convex_functions}, there exists $\tilde{\zeta}\in \Had{j}{n}$ such that
	$$\frac{\kappa_n\binom{n}{j}}{\kappa_k\binom{k}{j}} \int_{\Grass{k}{n}}  \oZZ{j}{\cR^{n-k}\zeta}^k(\proj_E u) \d E  = \oZZ{j}{\tilde\zeta}^n(u)$$
for every $u\in\fconvs$. We need to show that $\tilde \zeta=\zeta$.

Indeed, for $t\geq 0$, consider the function $u_t\in\fconvs$ defined in Lemma~\ref{le:calc_ind_bn_tx_theta_i}  and observe that 
$$\proj_E u_t (x_E)=t|x_E|+\ind_{B_E^k}(x_E)$$
for $x_E\in E$, where $B_E^k$ denotes the Euclidean unit ball in the $k$-dimensional space $E$.
It follows from Lemma~\ref{le:calc_ind_bn_tx_theta_i} that 
	\begin{equation*}
	\cR^{k-j}\cR^{n-k} \zeta =  \cR^{n-j} \tilde{\zeta}
	\end{equation*}
and therefore
$$\cR^{n-j} \zeta =  \cR^{n-j} \tilde{\zeta}.$$
Hence, Lemma~\ref{le:r_kln} implies that $\tilde{\zeta} = \zeta$.

\goodbreak
\subsection{Retrieving Intrinsic Volumes and Cauchy--Kubota formulas}
\label{se:ret_intr_vols}
The space, $\cK^n$, of convex bodies in $\R^n$ can be  embedded into the function space $\fconvs$ by identifying $K\in\cK^n$ with its indicator function $\ind_K\in\fconvs$. Similarly, we can embed $\cK^n$ into $\fconvf$ by identifying $K$ with its support function $h_K\in\fconvf$. As the following results show, the functional intrinsic volumes generalize the classical intrinsic volumes, and it is easy to retrieve the intrinsic volume $V_j$ on $\cK^n$ from both $\oZZ{j}{\zeta}^n$ on $\fconvs$ and $\oZZ{j}{\zeta}^{n,*}$ on $\fconvf$.

\begin{proposition}
\label{prop:retrieve_intrinsic_volumes}
If $\,0\leq j \leq n-1$ and $\zeta\in\Had{j}{n}$, then
$$\oZZ{j}{\zeta}^n(\ind_K) = \kappa_{n-j}  \cR^{n-j}\!\zeta(0)\,V_j(K)$$
or equivalently
$$\oZZ{j}{\zeta}^n(\ind_K) = (n-j) \kappa_{n-j}  \lim_{s\to 0^+} \int_s^{\infty} t^{n-j-1}\zeta(t) \d t\,V_j(K)$$
for every $K\in\cK^n$.
If $\zeta\in\Had{n}{n}$, then 
$$\oZZ{n}{\zeta}^n(\ind_K) = \zeta(0)\, V_n(K)$$ 
for every $K\in\cK^n$.
\end{proposition}
\begin{proof}
Let $K\in\cK^n$ be given and $0\leq j\leq n$. It follows from \eqref{eq:proj_fct_lvl_set} that $\proj_E \ind_K = \ind_{\proj_E K}$ and thus,
$$\int_{\dom(\proj_E \ind_K)} \alpha(|\nabla \proj_E \ind_K(x_E)|) \d x_E = \int_{\proj_E K} \alpha(0) \d x_E= \alpha(0)\, V_j(\proj_E K)$$
for every $\alpha\in C_c([0,\infty))$ and $E\in \Grass{j}{n}$, where integration is with respect to the Lebesgue measure on $E$. Hence, combining this, Theorem~\ref{cauchy_function}, Lemma~\ref{Rtol} and \eqref{cauchy_kubota}, we obtain
$$
\oZZ{j}{\zeta}^n(\ind_K) = \frac{\kappa_n}{\kappa_j} \binom{n}{j} \cR^{n-j}\!\zeta(0) \int_{\Grass{j}{n}} V_j(\proj_E K) \d E = \kappa_{n-j}  \cR^{n-j}\!\zeta(0)\, V_j(K),
$$
which concludes the proof.
\end{proof}

\goodbreak
Since $\ind_K^*=h_K$ for $K\in\cK^n$, we immediately obtain the following dual statement.

\begin{proposition}
\label{prop:retrieve_intrinsic_volumes_dual}
If $\,0\leq j \leq n-1$ and $\zeta\in\Had{j}{n}$, then
$$\oZZ{j}{\zeta}^{n,*}(h_K)=\kappa_{n-j}  \cR^{n-j}\!\zeta(0)\, V_j(K)$$
or equivalently
$$\oZZ{j}{\zeta}^{n,*}(h_K)=(n-j)\kappa_{n-j}\lim_{s\to 0^+} \int_s^{\infty} t^{n-j-1} \zeta(t) \d t\,V_j(K) $$
for every $K\in\cK^n$.
If $\zeta\in\Had{n}{n}$, then
$$\oZZ{n}{\zeta}^{n,*}(h_K)=\zeta(0)\, V_n(K)$$
for every $K\in\cK^n$.
\end{proposition}
\noindent
We remark that it is possible to prove Proposition~\ref{prop:retrieve_intrinsic_volumes} and Proposition~\ref{prop:retrieve_intrinsic_volumes_dual}, without using Theorem~\ref{cauchy_function}, by direct calculation.

Proposition~\ref{prop:retrieve_intrinsic_volumes} shows that our new Cauchy--Kubota formulas generalize the classical ones. In order to see this, let $0\leq j \leq k <n$ and choose $\alpha\in C_c([0,\infty))$ such that $\alpha(0)\neq 0$. Set $\zeta:=\cR^{-(n-j)}\alpha$ and note that by Lemma~\ref{le:r_kln} we have $\zeta\in\Had{j}{n}$. Choosing $u=\ind_K$ for some convex body $K\in\cK^n$ in Theorem~\ref{thm:cauchy_kubota_for_hessian_vals}, we obtain
\begin{align*}
\kappa_{n-j} \alpha(0) V_j(K) &= \oZZ{j}{\zeta}^n(\ind_K)\\
&= \frac{\kappa_n}{\kappa_k}\frac{\binom{n}{k}}{\binom{n-j}{k-j}} \int_{\Grass{k}{n}} \oZZ{j}{\cR^{n-k} \zeta}^k(\proj_E \ind_K) \d E\\
&=\frac{\kappa_n}{\kappa_k}\frac{\binom{n}{k}}{\binom{n-j}{k-j}} \int_{\Grass{k}{n}} \kappa_{k-j} (\cR^{k-j}\cR^{n-k} \cR^{-(n-j)}\alpha)(0) V_j(\proj_E K) \d E,
\end{align*}
where we  used Proposition~\ref{prop:retrieve_intrinsic_volumes} and the fact that $\proj_E \ind_K=\ind_{\proj_E K}$ for $E\in\Grass{k}{n}$, which follows from \eqref{eq:proj_fct_lvl_set}. Since $\cR^{k-j}\cR^{n-k} \cR^{-(n-j)}\alpha=\alpha$ and $\alpha(0)\neq 0$, we therefore obtain
$$\frac{\kappa_{n-j}}{\kappa_{k-j}}\binom{n-j}{k-j} V_j(K)=\frac{\kappa_n}{\kappa_k} \binom{n}{k} \int_{\Grass{k}{n}} V_j(\proj_E K) \d E.$$
Note that the special case $j=k$ is just \eqref{cauchy_kubota}.

\subsection{Non-negative and Monotone Valuations}
Theorem~\ref{cauchy_function} allows us to easily answer the question under which conditions on $\zeta\in\Had{j}{n}$ the valuation $\oZZ{j}{\zeta}^n$ is non-negative.

Let $1\leq j \leq n-1$ and $\zeta\in\Had{j}{n}$. Recall that $\alpha\in C_c([0,\infty))$ is given by
$$\alpha(s)= \kappa_{n-j} \big(  s^{n-j}\zeta(s)+(n-j)\int_s^{\infty}  t^{n-j-1} \zeta(t)\d t\big) = \kappa_{n-j} \cR^{n-j}\zeta(s)$$
for $s>0$. Since
$$\oZZ{j}{\zeta}^n(u)= \frac{\kappa_n}{\kappa_j\kappa_{n-j}} \binom{n}{j}
\int_{\Grass{j}{n}}  \int_{\dom (\proj_E u)} \alpha(|\nabla \proj_E u(x_E)|) \d x_E \d E$$
for every $u\in\fconvs$, it is easy to see that if $\alpha$ is non-negative, then so is $\oZZ{j}{\zeta}^n$.

\goodbreak
Conversely, assume that $\oZZ{j}{\zeta}^n(u)\geq 0$ for every $u\in\fconvs$. By Lemma~\ref{le:calc_ind_bn_tx_theta_i} we now have
$$0\leq \oZZ{j}{\zeta}^n(u_t)=\frac{\kappa_n}{\kappa_{n-j}} \binom{n}{j} \, \alpha(t)$$
for every $t\geq 0$. Thus, $\alpha$ needs to be non-negative.

\goodbreak
In the cases $j=0$ and $j=n$,  non-negativity is easy to describe. Thus, we have  shown the following result.

\begin{proposition}
For $j=0$, the valuation $\oZZ{j}{\zeta}^n$ is non-negative if and only if $\,\lim_{s\to 0^+} \int_{s}^{\infty} t^{n-1} \zeta(t) \d t \geq 0$. For $j=n$, the valuation $\oZZ{j}{\zeta}^n$ is non-negative if and only if $\zeta$ is non-negative. For $1\leq j \leq n-1$, the valuation $\oZZ{j}{\zeta}^n$ is non-negative if and only if
$$s^{n-j}\zeta(s)+(n-j)\int_s^{\infty}  t^{n-j-1} \zeta(t)\d t \geq 0$$
for every $s>0$.
\end{proposition}

A valuation $\oZ\colon\fconvs\to\R$ is \emph{increasing}, if $\oZ(u_1)\le\oZ(u_2)$ for all $u_1,u_2\in\fconvs$ such that $u_1\le u_2$. It is \emph{decreasing} if $\oZ(u_1)\ge\oZ(u_2)$ for all $u_1,u_2\in\fconvs$ such that $u_1\le u_2$. It is \emph{monotone} if it is decreasing or increasing. 

\begin{proposition} If $\,\oZ$ is a continuous, epi-translation invariant, and monotone valuation on $\fconvs$, then $\oZ$ is constant.
\end{proposition}

\begin{proof} Without loss of generality we assume that $\oZ$ is increasing. By Lemma~\ref{le:C2p_dense} and the continuity of $\oZ$, it is sufficient to prove that $\oZ(u_1)=\oZ(u_2)$ for every $u_1,u_2\in\fconvs$ such that $\dom(u_1)=\dom(u_2)=\R^n$.
Fix two such functions $u_1,u_2\in\fconvs$. For $r>0$, let $B_r:=\{x\in\R^n\colon |x|\le r\}$ and set
$$
u_{1,r}=u_1+\ind_{B_r},\quad\quad u_{2,r}=u_2+\ind_{B_r}.
$$
As $u_1$ and $u_2$ are continuous in $B_r$, there exists $\gamma>0$ such that
$$
u_{2,r}(x)-\gamma\le u_{1,r}(x)\le u_{2,r}(x)+\gamma
$$
for every $x\in\R^n$. From the epi-translation invariance and monotonicity of $\oZ$, we deduce
$$
\oZ(u_{1,r})=\oZ(u_{2,r}),
$$
and this equality holds for every $r>0$. On the other hand $u_{1,r}$ and $u_{2,r}$ epi-converge to $u_1$ and $u_2$, respectively, as $r\to\infty$. The continuity of $\oZ$ implies that $\oZ(u_1)=\oZ(u_2)$.
\end{proof}

\noindent We remark that monotone functionals on convex functions that are epi-additive, that is, additive with respect to infimal convolution, were classified in \cite{Rotem22}. Rigid motion invariant and monotone valuations (that are not necessarily epi-translation invariant) were studied in \cite{CavallinaColesanti}.

\subsection*{Acknowledgments}
The authors thank the referee for their careful reading and helpful remarks.
M.~Ludwig was supported, in part, by the Austrian Science Fund (FWF):  P 34446.
F. Mussnig was supported, in part, by the European Research Council (ERC) under the European Union's Horizon 2020 research and innovation programme (grant agreement {No.~770127}) and, in part, by the Austrian Science Fund (FWF): J 4490-N.

\end{document}